\documentclass[11pt]{amsart}
\usepackage{amscd,amssymb,amsmath,enumerate}
\usepackage{amsmath,amsthm}
\usepackage[all]{xy}
\usepackage{graphicx}
\usepackage{graphics}
\usepackage{latexsym}
\usepackage[all]{xy}
\usepackage{psfrag}
\xyoption{matrix} \xyoption{arrow}
\usepackage{parskip}

\newtheorem{proposition}{Proposition}[section]
\newtheorem{theorem}[proposition]{Theorem}
\newtheorem{lemma}[proposition]{Lemma}

\newtheorem{definition}[proposition]{{Definition}}

\newtheorem{remark}[proposition]{{Remark}}

\newcommand{\cA}{{\mathcal A}}
\newcommand{\cB}{{\mathcal B}}

\newcommand{\cM}{{\mathcal M}}

\newcommand{\cC}{{\mathcal C}}

\newcommand{\cS}{{\mathcal S}}

\renewcommand{\aa}

\newcommand{\Hom}{\operatorname{Hom}\nolimits}

\newcommand{\bc}{Brauer configuration}
\newcommand{\bca}{\bc\ algebra}

\usepackage{color}
\definecolor{candyapplered}{rgb}{1.0, 0.03, 0.0}
\makeatletter
\def\thm@space@setup{%
  \thm@preskip=0.7cm \thm@postskip=0.3cm
}
\makeatother

\begin{document}

\date{today}
\title[Quotients of symmetric special multiserial algebras]
{Special multserial algebras are quotients of
symmetric special multiserial algebras}

\author[Green]{Edward L.\ Green}
\address{Edward L.\ Green, Department of
Mathematics\\ Virginia Tech\\ Blacksburg, VA 24061\\
USA}
\email{green@math.vt.edu}
\author[Schroll]{Sibylle Schroll}
\address{Sibylle Schroll\\
Department of Mathematics \\
University of Leicester \\
University Road \\
Leicester LE1 7RH \\
United Kingdom}
\email{schroll@le.ac.uk }

\thanks{This work was supported through the Engineering and Physical Sciences Research Council, grant
number EP/K026364/1, UK and by the University of Leicester in form of a study leave for the second author.}

\begin{abstract} In this paper we give a new definition of symmetric special multiserial algebras in terms of \emph{defining cycles}. As a consequence, we show 
that every special multiserial algebra is a quotient of a symmetric special multiserial algebra. 
\end{abstract}
\date{\today}
\maketitle


\section{Introduction}

 A major breakthrough in representation theory of finite dimensional algebras is the classification of algebras in terms of their representation type. This is either finite, tame or wild \cite{D}. Algebras of finite representation type have only finitely many isomorphism classes of indecomposable modules, the infinitely many indecomposable modules of a tame algebra can be parametrized by one-parameter families whereas the representation theory of a wild algebra contains that of the free algebra in two generators and so in some sense contains that of any finite dimensional algebra. Thus no hope of a parametrization of the isomorphism classes of the indecomposable modules can exist. 

For this reason, algebras of finite and tame representation type have been the focus of much of the representation theory of finite dimensional algebras. An important family of tame algebras are special biserial algebras defined in \cite{SW}. This class contains many of the tame group algebras of finite groups and tame subalgebras of group algebras of finite groups \cite{E}, gentle algebras, string algebras and symmetric special biserial algebras \cite{WW}, also known as Brauer graph algebras \cite{R, S}, algebras of quasi-quaternion type \cite{La} and the intensely studied Jacobian algebras of surface triangulations with marked points in the boundary arising in cluster theory \cite{ABCP}.  

The strength of the well-studied representation theory of special biserial algebras,  derives from the underlying string combinatorics. Namely, by \cite{GP, Ri, WW} every indecomposable non-projective module over a special biserial algebra is a string or band module. Not only does this give rise to a formidable tool for calculations and proofs but it also shows that special biserial algebras are of tame representation type. 

Special multiserial algebras, defined in \cite{VHW}, are in general of wild representation type and as a consequence their indecomposable modules cannot be classified in a similar way. 
It is therefore remarkable that many of the results that are known to hold for special biserial algebras still hold for special multiserial algebras. For example, a very surprising fact about the indecomposable modules of these wild algebras was shown in \cite{GS2}. Namely, the indecomposable modules over 
a special multiserial algebra are multiserial, that is, their radical is either 0 or a sum of uniserial submodules. Thus generalizing the analogous result for special biserial algebras \cite{SW}. However, given the absence of string
combinatorics in the multiserial case, the proof is built on an entirely different strategy. The same holds true for the result and proofs in this paper. 

We start by giving the definition of an algebra \emph{defined by cycles}. This definition is built on the notion of a \emph{defining pair}. We show that such an algebra is symmetric special multiserial and that conversely, every symmetric special multiserial algebra is an algebra defined by cycles.  Note that in this context a symmetric algebra is an algebra over a field  endowed with a symmetric linear form with no non-zero left ideal in its kernel.  Symmetric algebras play an important role in representation theory and many examples of well-known algebras are symmetric such as group algebras of finite groups or Hecke algebras. 

Given the new definitions of defining pairs and  algebras defined by cycles we show that we can construct a defining pair for every special multiserial algebra $A$ and that $A$ is a quotient of the corresponding algebra defined by cycles. Thus we prove that every special multiserial algebra is a quotient of a symmetric special multiserial algebra. This result is an analogue of the corresponding result for special biserial algebras \cite{WW}.  Moreover, the 
special biserial case follows from our  result omitting thus the need for the string combinatorics on which the proof in \cite{WW} is based.

\section{Preliminaries}

We let  $K$ denote a field and $Q$ a quiver.  An ideal $I$ in  the path algebra $KQ$ is \emph{admissible}
if $J^N\subseteq I\subseteq J^2$ for some $N\ge 2$, where $J$ is the ideal
in $KQ$ generated by the arrows of $Q$.
We begin by recalling the definition of a special multiserial algebra. 
 A $K$-algebra is a 
\emph{special multiserial algebra} if it is Morita equivalent to a quotient of a  path algebra,
$KQ/I$, by an admissible ideal $ I$  which satisfies the following condition:
\[\quad (M) \quad\text{ For every arrow }a\in Q\text{ there is at most one arrow }b\in Q\text{ such that }
ab\notin I\text{ and}\]
\[\quad\quad\text{ there is at most one arrow }c\in Q\text{ such that }ca\notin I\]

Throughout this paper we assume that all algebras are indecomposable. First we treat the radical square zero case.  We note that if $A=KQ/I$ with $I$
admissible is such that the
Jacobson radical of $A$  squares to zero, then $I=J^2$ 
and $A$ is a special multiserial algebra since all paths of length 2 are in $I$.   Thus, if $A$ is a radical square zero 
algebra, then we wish to show
that $A$ is the quotient of a symmetric special multiserial
algebra.  For this we recall that an algebra
$\Lambda =KQ/I$, with $I$ generated by a set of paths
of length $2$ is called \emph{gentle} if
\begin{enumerate}
\item[(G1)] $\Lambda$ is a special multiserial algebra, and,
\item[(G2)] at each vertex $v$  there are at most 2 arrows ending
at $v$ and at most 2 arrows starting at $v$.
\end{enumerate}
Recall from \cite{GS3} that an algebra $KQ/I$ is almost gentle if it is special multiserial and if $I$ can be generated by paths of length 2.  

\begin{proposition}\label{prop-r2-0} Assume that $A=KQ/J^2$ where 
$J$ is the ideal in $KQ$ generated by the arrows of $Q$.
Then $A$ is the quotient of a symmetric special
multiserial algebra.  In particular, the trivial extension
$A \rtimes D(A)$  is a symmetric special multiserial
algebra whose radical cube is zero, where $D(-)$ denotes
the duality $\Hom_K(-, K)$.
\end{proposition} 

\begin{proof}
That $A$ is an almost gentle algebra is clear.
By \cite{GS3} Theorem 4.3, we have that $\Lambda=A\rtimes D(A)$ is a symmetric
special multiserial algebra.   
The algebra $A$ is clearly a quotient of $\Lambda$
and, since
$I=J^2$, the radical cube of $\Lambda$ is zero .
\end{proof}

We introduce notation and definitions needed for what follows.
We say that a nonzero element $x\in KQ$ is \emph{uniform}
if there exists vertices $v$ and $w$ in $Q$, corresponding to idempotents $e_v$ and $e_w$ in $KQ$ such that
$x= e_vxe_w$.  If $a$ and $b$ are arrows in $Q$, we let $ab$ denote
the path consisting of the arrow $a$ followed by the arrow $b$. If $p$ is a path in $Q$, then the start vertex of $p$ is denoted
$\mathfrak s(p)$ and the end vertex of $p$ is denoted $\mathfrak t(p)$.
 A cycle is a path $p$ such that $\mathfrak s(p)= \mathfrak t(p)$. If $C = a_1 ... a_n$ is a cycle in $Q$, then a cyclic permutation of $C$ is any cycle of the form $a_i a_{i+1} \ldots a_n a_1 \ldots a_{i-1}$ for all $ 1 \leq i \leq n$.   We say that a cycle in $Q$  is \emph{simple} if it has no repeated
arrows.  Note that the number of simple cycles is always finite.   If $C$ is a cycle in $Q$ and $p$ is a  path, we say
\emph{$p$ lies in  $C$} if $p$ is a subpath of $C^s$, for
some $s\ge 1$.  If $p$ is a path in $Q$, the \emph{length of
$p$}, denoted $\ell(p)$, is the number of arrows in $p$.

\section{Defining pairs}

In this section we give a method for constructing symmetric
special multiserial algebras. 
Suppose that $Q$ is a quiver.   We say the pair $(\cS,\mu)$
is a \emph{defining  pair in $Q$} if  $\cS$ is a set of simple cycles in $Q$
and $\mu\colon\cS\to \mathbb Z_{>0}$ which satisfy the following conditions:
\begin{enumerate}
\item[(D0)] If $C$ is a loop at a vertex $v$ and $C\in\cS$, then $\mu(C)> 1$.
\item[(D1)] If a simple cycle is in $\cS$, every cyclic permutation of the cycle is in $\cS$.
\item[(D2)] If $C\in\cS$ and $C'$ is a cyclic  permutation of $C$ then $\mu(C)=\mu(C')$.
\item[(D3)] Every arrow occurs in some simple cycle in $\cS$.
\item[(D4)]  If an arrow occurs in two cycles in  $ \cS$, the cycles are cyclic permutations
of each other.
\end{enumerate}

If $(\cS,\mu)$ is a defining pair in $Q$ then the $K$-algebra they define has
quiver $Q$ and ideal of relations generated by all relations of the following three
types:
\begin{enumerate}
\item[Type 1] $C^{\mu(C)}-{C'}^{\mu(C')}$,  if $C$ and $C'$ are cycles in $\cS$ at some 
vertex $v\in Q_0$.
\item[Type 2]  $ C^{\mu(C)}a$, if $C\in\cS$ and $a$ is the first arrow in $C$.
\item[Type 3]  $ab$, if $a,b\in Q_1$ and $ab$ 
is not a subpath of 
any $C\in\cS$.
\end{enumerate} 

The algebra $A=KQ/I$, where $I$ is generated by all relations of Types 1, 2, and 3, is 
called \emph{the algebra defined by $(\cS,\mu)$} and we call $A$ an \emph{algebra defined by cycles}.  We note that  some of the generators  of Types 1,2, and 3
are in general redundant. 

\begin{theorem}\label{defining  pair}Let $Q$ be a quiver, $K$ a field, and $(\cS,\mu)$ a
defining pair for $Q$.  Let $A=KQ/I$ be the algebra defined by $(\cS,\mu)$.   Then $A$
is a symmetric special multiserial algebra.

\end{theorem}

{\it Proof.}
We begin by showing that $I$ is an admissible  ideal.  Clearly, $I$ is contained
in the ideal generated by  paths of length 2.  Let $N=\max\{\mu(C)\ell(C)\mid 
{C\in\cS}\}$.
 We claim that all  paths of length $N+1$ are in $I$. 
 Let $p$ be such
a  path.  If there are arrows $a$ and $a'$ such that $aa'$ is a subpath of $p$ and $aa'$ is a
Type 3 relation, then $p\in I$.  Suppose that $p$ contains no Type 3 relations. Then (D3) and the definition of $N$ imply that there is a simple cycle $C\in\cS$ and an arrow $b$ so that either 
$bC ^{\mu(C)}$ or $C^{\mu(C)}b$ is a subpath of $p$.  
 We suppose first that $C^{\mu(C)}b$ is a subpath of $p$.
  Then $C^{\mu(C)}b$ either 
is a Type 2 relation or $Cb$ contains a Type 3 relation.   Finally, suppose that $bC^{\mu(C)}$ is a subpath
of $p$.  If $b$ is the last arrow in $C$, then $bC^{\mu(C)}$ is a Type 2 relation using (D1)  and (D2).  If
$b$ is not the last arrow in $C$, then $ba$ is a Type 3 relation where $a$ is the first arrow of $C$.

Next we show that $A$ is a special multiserial algebra.  Let $a$ be an arrow.  Suppose that
$ab$ is a path of length 2 in $Q$.   By (D3) and (D4), $a$ is in a cycle $C$ in $\cS$ that is unique up to 
cyclic permutation.  Either $b$ is the unique arrow such that $ab$ lies in $C$ or $ab$ is a Type  3 relation
and hence in $I$.  Thus, there is at  most one arrow $b$ such that $ab\notin I$.  Similarly, there
is at  most one arrow $c$ such that $ca\notin I$ and we see that $A $  is a special  multiserial
algebra.

Finally we show that $A$ is a symmetric algebra. 
We let  $\pi\colon KQ\to A$ denote the canonical surjection.
  Define $f\colon KQ\to K$ as follows.
If $p$ is a path in $Q$, define  
$$f(p)=\begin{cases} 1, &\text{if }p=C^{\mu(C)}\text{ for some } 
C\in \cS\\ 0, &\text{otherwise. }
\end{cases}$$  Linearly extend $f$ to $KQ$.   The reader may check that
$f$ induces a linear map $\bar f\colon A\to K$.  Let $\cB$ be the set of  paths
in $Q$.  All sums will have only a finite number of nonzero terms.
Note that if $x=\sum_{p\in\cB} \alpha_pp\in KQ$ with $\alpha_p\in K$, then $f(x)=\sum_{C\in\cS}\alpha_{C^{\mu(C)}}$.

First we show that 
if $\lambda,\lambda'\in A$, then
$\bar f(\lambda\lambda')=\bar f(\lambda'\lambda)$.   
Let $x=\sum_{p\in\cB}\alpha_pp\in KQ$ and $y=\sum_{q\in\cB}\beta_qq\in KQ$
such that $\pi(x)=\lambda$ and $\pi(y)=\lambda'$.
Then
\[
f(xy)=f((\sum_{p\in\cB}\alpha_pp)(\sum_{q\in\cB}\beta_qq))=
\sum_{p\in\cB}\sum_{q\in\cB}\alpha_p\beta_qf(pq)=\sum_{pq=C^{\mu(C)}\text{ for }C\in\cS}\alpha_p\beta_q
\]
and
\[
f(yx)=f((\sum_{q\in\cB}\beta_qq)(\sum_{p\in\cB}\alpha_pp))=
\sum_{q\in\cB}\sum_{p\in\cB}\beta_q\alpha_pf(qp)=\sum_{qp=C^{\mu(C)}\text{ for }C\in\cS}\beta_q\alpha_p
\]

 In the above equations for all $p, q \in \cB$ for which $pq \neq 0$, we have that $pq$ is a cyclic permutation of $qp$. It then follows that   $f(xy)=f(yx)$ and hence $\bar f(\lambda\lambda')=\bar f(\lambda'\lambda)$.

Finally
we claim that $\ker(\bar f)$ contains 
no nonzero left or right ideals.   We start by proving that there are no right ideals in $\ker \bar f$.  
Suppose that $\mathfrak I$ is a right  ideal of $A$ contained
in the kernel of $\bar f$.  Assume $\mathfrak I\ne (0)$ and let $\lambda\in\mathfrak I$ with
$\lambda\ne 0$.  Then $\lambda=\pi(\sum_{p\in\cB}\alpha_pp)$ where $\alpha_p\in K$ and
all but a finite number of $\alpha_p\ne 0$.  Without loss of generality, we  may assume that
if $\alpha_p\ne 0$, then $p\notin I$.  First suppose that there is a path $p^*\notin I$ such that
$\alpha_p^*\ne 0$ and $p^*$ is not of the form $C^{\mu(C)}$ for any $C\in\cS$.   Then there is
a unique $C\in\cS$ and path $q$ such  that $p^*q=C^{\mu(C)}$.   By (D1)-(D4), if $p'\ne p^*$, 
then $p'q$ is not of the form $C'^{\mu(C')}$ for any $C'\in\cS$.
Hence \[\bar f(\lambda\pi(q))=\bar f(\pi((\sum_{p\in\cB}\alpha_pp)q))=f(\sum_{p\in\cB}\alpha_ppq)
=\alpha_{p^*}\ne 0.\]
But this contradicts $\mathfrak I\subseteq \ker(\bar f)$ since $\lambda \pi(q)\in\mathfrak I$.
The proof that there a no left ideals in $\ker(\bar f)$ is similar to the proof for right ideals.

Thus we may assume that if $\alpha_p\ne 0$ then $p=C^{\mu(C)}$ for some $C\in\cS$. So we have $\lambda = \pi (\sum_{v \in Q_0} \sum_{C \in \cS_v} \alpha_{C^{\mu(C)}}C^{\mu(C)})$, where $\cS_v =\{C\in\cS\mid C \text{ is a cycle at } v\}$ for any $v \in Q_0$.
Since $\lambda\ne 0$, we see there is some vertex $v$ such that ${ e_v}\lambda e_v\ne 0$. Then
\[0=\bar f({e_v}\lambda { e_v})=\bar f(\pi(\sum_{C\in\cS_v}\alpha_{C^{\mu(C)}}C^{\mu(C)}))=
f(\sum_{C\in\cS_v}\alpha_{C^{\mu(C)}}C^{\mu(C)})=\sum_{C\in\cS_v}\alpha_{C^{\mu(C)}}.\]
Choose some $C^*\in\cS_v$.   Then, using that, for $C\in\cS_v$, $(C^{\mu(C)}-{C^*}^{\mu(C^*)})$ is a Type 1
relation and $\sum_{C\in\cS_v}\alpha_{C^{\mu(C)}}=0$ , we see that
\[  e_v\lambda  e_v=\pi(\sum_{C\in\cS_v}\alpha_{C^{\mu(C)}}C^{\mu(C)})-
(\sum_{C\in\cS_v}\alpha_{C^{\mu(C)}})\pi({C^*}^{\mu(C^*)})=\]\[
\sum_{C\in\cS_v}\alpha_{C^{\mu(C)}}\pi(C^{\mu(C)}-{C^*}^{\mu(C^*)})=0,
\]
contradicting $e_v\lambda e_v\ne 0$.  This completes the proof.

The converse of the above Theorem is also true.  The Theorem below is not
used in the remainder of the paper and we only sketch the
proof.  The sketch below assumes
knowledge of Brauer configuration algebras found in \cite{GS1}.

\begin{theorem}\label{symm spec multiserial has defining pair} Let $A=KQ/I$ be an indecomposable symmetric
special multiserial algebra with Jacobson radical squared nonzero, where $I$ is an admissible ideal in $KQ$.  Then there is a defining
pair $(\cS,\mu)$ in $Q$ such that the  algebra defined by $(\cS,\mu)$ is isomorphic to $A$.
\end{theorem}

{\it Proof.}
By \cite{GS2}, we may assume that $A$ is the \bca\ associated to a \bc\ $\Gamma=
(\Gamma_0,\Gamma_1,\mu, \mathfrak o)$.   Let $\cS$ be the set of special
cycles in the quiver of $A$.   If $C$ is a special $\alpha$-cycle for some $\alpha\in\Gamma_0$,
define $\mu^*(C)=\mu(\alpha)$.  Using the properties of special cycles of a \bca,  the
reader may check that $(\cS,\mu^*)$ is a defining pair in $Q$.  It is straightforward 
to see that the \bca\ $A$ is isomorphic to the algebra defined by $(\cS,\mu^*)$.

\section{quotients}

For the remainder of this section, we let $A=KQ/I$ be a special multiserial algebra, that is $I$ satisfies
condition (M).  
We introduce two functions associated to $A$ which play a central role.
Let $\diamond$ be some element not in $Q_1$ and 
set $\cA=Q_1 \cup \{\diamond\}$.  Define
$\sigma\colon Q_1\to \cA$ and $\tau\colon Q_1\to \cA$ by
\[\sigma(a)=\begin{cases} b& \text{ if } ab\notin I\\
                \diamond &\text{ if }ab\in I \text{ for all }b\in Q_1\end{cases}
\]
 and
\[\tau(a)=\begin{cases} c& \text{ if } ca\notin I\\
                \diamond &\text{ if }ca\notin I \text{ for all }c\in Q_1\end{cases}.\]
where $a,b,c\in Q_1$. From the definition of a special multiserial algebra, we see that
these functions are well-defined. Since $A$ is finite dimensional,  one of two things
occur for $\sigma$ when repeatedly applied to an arrow $a\in Q_1$.  
Either there is a smallest  positive integer $m_a$ such that 
$\sigma^{m_a}(a)=\diamond$ or there is a smallest  positive
integer $\hat m_a$ such that $\sigma^{\hat m_a}(a)=a$.  Similarly, either there is a
smallest positive integer $n_a$ such that $\tau^{n_a}(a)=\diamond$ or there
is a positive integer $\hat n_a$ such that $\tau^{\hat n_a}(a)=a$.  

We list some basic properties of $\sigma$ and $\tau$.
\begin{enumerate}
\item[B1] If $\sigma(a)\in Q_1$, then $\tau\sigma(a)=a$.
\item[B2] If $\tau(a)\in Q_1$, then $\sigma\tau(a)=a$.
\item[B3]For $a\in Q_1$,  $m_a$ exists if and only if $n_a$ exists.
\item[B4] For $a\in Q_1$,  $\hat m_a$ exists if and only if $\hat n_a$ exists.
\end{enumerate}

Along with $\sigma$ and $\tau$, we need one more concept.  We say a
path $M=a_1a_2\cdots a_r$, with $a_i\in Q_1$ is \emph{$(\sigma,\tau)$-maximal} if
$\sigma(a_r)=\diamond=\tau(a_1)$.   Let $\cM$ denote the set of
$(\sigma,\tau)$-maximal paths.

Suppose $a\in Q_1$ is such that $\hat m_a$ exists.  Then we have a simple oriented
cycle in $Q$, denoted $C_a$ such that $C_a =a \sigma(a) \sigma^2(a)\cdots
 \sigma^{\hat m_a-1}(a)$
since $\sigma^{\hat m_a}(a)=a$ implies that $\mathfrak s(C_a)$ = $\mathfrak t(C_a)$.  It is easy to check that $C_a$ is a simple cycle.
We let the set of such simple cycles be denoted by  $\cC$.
Note that if $C\in \cC$, then every cyclic permutation of $C$ is in $\cC$.

Now suppose that $a\in Q_1$ is such that $m_a$ exists.   Then 
\[M_a=\tau^{n_a-1}(a)\tau^{n_a-2}(a)\cdots\tau(a)a\sigma(a)\cdots\sigma^{m_a-1}(a)\]
is a $(\sigma,\tau)$-maximal path in which $a$  occurs.

The next lemma lists some basic properties of the above constructions.
The proof is straightforward and left to the reader.

\begin{lemma}\label{max paths and cycles}Let $A=KQ/I$ be a special multiserial algebra
and let $a\in Q_1$.  Then
\begin{enumerate}
\item Either $a$ occurs in the $(\sigma,\tau)$-maximal path $M_a$ or the simple cycle $C_a$ but not both.
\item The arrow $a$ occurs in at most one $(\sigma,\tau)$-maximal path.
\item  If $a$ is an arrow in a simple cycle $C_b\in\cC$, for some arrow $b$,
 then $C_b$ is a cyclic permutation
of $C_a$.
\item The length of $C_a$, if it exists, is $\hat m_a =\hat n_a$.
\item  If $M\in\cM$ is a maximal  path, then $M$ has no repeated arrows.
\item If $M\in\cM$ and $a$ is an arrow in  $M$, then $M=M_a$.
\item  If $a$ occurs in the $(\sigma,\tau)$-maximal path $M_a$ then the length of $M_a$ is
$m_a+n_a -1$.

\end{enumerate}
\end{lemma}

We now construct a new quiver, $Q^*$, from $Q$. Set $Q^*_0=Q_0$.
For each $M\in \cM$,  let $a_M$ be an arrow from $\mathfrak t(M)$ to
$\mathfrak s(M)$.  Set $Q^*_1=Q_1\cup\{a_M\mid M\in \cM\}$.  
Note that $Ma_M$ is a simple cycle in $Q^*$ at $\mathfrak s(M)$.
Let $\cM^*$ denote the set of cycles in $Q^*$ consisting of all
cyclic permutations of the $Ma_M$, for $M\in\cM$.

Since $Q$ is a subquiver of $Q^*$, we will freely view  paths 
and cycles in $Q$ as
paths or cycles in $Q^*$.  Let \[\cS=\{C\in \cC\}\cup\cM^*,\] viewed
as a set of simple cycles in $Q^*$.
Next, since $I$ is admissble, there is a smallest positive integer $N$,
$N\ge 2$, such that
all paths of length $N$ or larger, are in $I$.  Define
$\mu\colon\cS\to \mathbb Z_{>0}$ by $\mu(C)=N$ for
all $C\in\cS$ 

\begin{proposition}\label{defining set} Keeping the above notation,
$(\cS,\mu)$ is a defining  pair in  $Q^*$.
\end{proposition}

{\it Proof.}
Since $N\ge 2$ and we see that (D0) holds.
Since $\mu$ is constant on $\cS$, (D2) holds.  It is immediate
that (D1) holds.  If $a$ is an arrow in $Q$, then $a$ occurs in
some $C\in\cC$ or $M\in \cM$.  It is easy to see that for all $M\in \cM$, the arrow $a_M$
occurs in some cycle in $\cM^*$.  Thus, every arrow in $ Q^*$ 
occurs in some cycle in $\cS$ and (D3) holds.  By Lemma \ref{max
paths and cycles}, and our construction, (D4) holds and the proof
is complete.

Let $A^*=KQ^*/I^*$ be the algebra defined by $(\cS,\mu)$.   By Theorem \ref{defining pair}, $A^*$ is a symmetric special multiserial algebra which we
call the \emph{symmetric special multiserial algebra associated to $A$}.

\begin{theorem} \label{quotient theorem}Let $A$  be a special multiserial algebra and $A^*$ be
the symmetric special multiserial algebra associated to $A$  defined by $(\cS,\mu)$.
Then $A$ is a quotient of $A^*$.

\end{theorem}

{\it Proof.} To define $F\colon KQ^* \to KQ$ we use the universal mapping
property of path algebras.  
That is, let $F(e_v)$, for all $v\in Q^*_0$, be a full  set of orthogonal idempotents in $KQ$, $F(a_M)=0$, for all $M\in\cM$, and
$F (a)=a$ for all the remaining arrows { in $Q^*$}.  The $K$-algebra
homomorphism $F$ is clearly surjective.  The homomorphism $F$
will induce a surjection $\hat F\colon A^*\to A$ if
$F(I^*)\subseteq I$.

We prove $F(I^*)\subseteq I$ by showing 
that $F$ applied the generators of $I^*$ of
Types 1,2, and 3 are in $I$.  If  $C$ and $C'$ are in $\cS$ then consider
$C^{\mu(C)}-C'^{\mu(C')}$ .
If $C$ (or $C'$) contains an arrow of the form $a_M$, for some $M\in\cM$,
then 	$F$ sends $C^{\mu(C)}$ (or  $C'^{\mu(C')}$) to 0 which is in $I$.
If  $a_M$ does not occur in $C$ (or in  $C'$) then $C^{\mu(C)}$
(or $C'^{\mu(C')})$ has length greater than or equal to $N$ since $\mu$ has
constant value $N$.  Recalling that paths of length greater or equal to $N$
in $KQ$ are in $I$, we conclude that $F$ applied to a Type 1 relation
is in $I$.

A similar argument works for Type 2 relations.

Finally, suppose that $ab$ is a Type 3 relation.  If either $a$ or $b$ is
an arrow of the form $a _M$, for some $M\in\cM$, then $F(ab)=0\in I$.
Suppose that neither $a$ nor $b$ is of the form $a_M$. 
Then $F(ab)=ab$.   Since $ab$ does
not live on any $C\in\cS$,  we see that $ab$ does not
 live on any $C\in \cC$,  where $\cC$ is the set of simple cycles of $A$ as defined at the beginning of this section, nor is  $ab$ a subpath of any $M$ in $\cM$.
It follows that $\sigma(a)\ne b$.  But then $ab\in I$ and we are done.

\bibliographystyle{plain}

\end{document}